\documentclass[11pt,reqno]{amsart}
\usepackage[active]{srcltx}  % for the inverse search
\usepackage{amssymb,amsmath}
\usepackage[pdftex]{graphicx}
\usepackage{amssymb,amsmath}
\usepackage{verbatim}        % comment the following line to get rid
\renewenvironment{comment}{\begin{quote}\sf}{\end{quote}}

\def\freq{{\rm freq}}
\def\sig{\sigma}
\def\Int{\mbox{\rm Int}}
\def\nula{\nu_\lambda}

\newcommand{\diam}{\mbox{\rm diam}}

\newcommand{\be}{\begin{eqnarray}}
\newcommand{\ee}{\end{eqnarray}}
\newcommand{\supp}{\mbox{\rm supp}}

\newcommand{\card}{\#}
\def\wtil{\widetilde}

\newcommand{\half}{\frac{1}{2}}

\newcommand{\e}{{\varepsilon}}

\newcommand{\R}{{\mathbb R}}

\newcommand{\Nat}{{\mathbb N}}
\def\N{\Nat}

\newcommand{\Ak}{{\mathcal A}}

\newcommand{\Pk}{{\mathcal P}}

\newcommand{\Ek}{E}

\newcommand{\Uk}{{\mathcal U}}

\newcommand{\dist}{\mbox{\rm dist}}

\newcommand{\lam}{\lambda}

\newcommand{\gam}{\gamma}

\newcommand{\Leb}{{\mathcal L}}

\newcommand{\iii}{\mathbf{i}}
\newcommand{\jjj}{\mathbf{j}}
\newcommand{\kkk}{\mathbf{k}}

\def\ov{\overline}
\newtheorem{theorem}{Theorem}[section]

\newtheorem{lemma}[theorem]{Lemma}
\newtheorem{cor}[theorem]{Corollary}
\newtheorem{prop}[theorem]{Proposition}

\theoremstyle{definition}

\theoremstyle{definition}
\newtheorem{remark}[theorem]{Remark}

\numberwithin{equation}{section}

\input epsf.sty

\begin{document}

\thispagestyle{empty}

\begin{abstract}
Let $\nula^p$ be the distribution of the random series $\sum_{n=1}^\infty i_n \lam^n$, where
$i_n$ is a sequence of i.i.d. random variables taking the values 0,1 with probabilities $p,1-p$. These measures are the well-known (biased) Bernoulli convolutions.

In this paper we study the multifractal spectrum of $\nula^p$ for typical $\lam$. Namely, we investigate the size of the sets
\[
\Delta_{\lam,p}(\alpha) = \left\{x\in\R: \lim_{r\searrow 0} \frac{\log \nula^p(B(x,r))}{\log r} =\alpha\right\}.
\]
Our main results highlight the fact that for almost all, and in some cases all, $\lam$ in an appropriate range, $\Delta_{\lam,p}(\alpha)$ is nonempty and, moreover, has positive Hausdorff dimension, for many values of $\alpha$. This happens even in parameter regions for which $\nula^p$ is typically absolutely continuous.
\end{abstract}

%%%%%%%%%%%%%%%%%%%%%%%%%%%%%%%%%%%%%%%%%%%%%%%%%%%%%%%%%%%%%%%%%%%%%%%%%%%%%%%%%%%%%

\title[Multifractal structure]
{{ Multifractal structure
of Bernoulli convolutions}}

\author{Thomas Jordan}
\address{Thomas Jordan\\
University Walk\\
Bristol BS8 1TW\\
UK
}
\email{Thomas.Jordan@bristol.ac.uk}

\author{Pablo Shmerkin}
\address{Pablo Shmerkin\\
         University of Manchester. Oxford Road\\
         School of Mathematics. Alan Turing Building\\
         Manchester M13 9PL\\
         UK }
\email{Pablo.Shmerkin@manchester.ac.uk}

\author{Boris Solomyak}
\address{
Boris Solomyak\\
Department of Mathematics, University of Washington\\
Box 354350\\
Seattle, WA 98195-4350\\
USA
}
\email{solomyak@math.washington.edu}

\thanks{P.S. acknowledges support from EPSRC grant EP/E050441/1 and the University of Manchester.
B.S. was supported in part by NSF grants DMS-0654408 and DMS-0968879}



\date{\today}

\maketitle

\thispagestyle{empty}

%%%%%%%%%%%%%%%%%%%%%%%%%%%%%%%%%%%%%%%%%%%%%%%%%%%%%%%%%%%%%%%%%%%%%%%%%%%%%%%%%%%%%

\section{Introduction} \label{sec:introduction}

\subsection{Background}

If $\mu$ is a measure on $\R^d$ (or, more generally, on a metric space), the \textit{local dimension} of $\mu$ at $x$ is defined as
\begin{equation} \label{eq:def-local-dim}
d(\mu,x) = \lim_{r\searrow 0} \frac{\log\mu(B(x,r))}{\log r},
\end{equation}
provided the limit exists. (Here, and throughout the paper, $B(x,r)$ denotes the open ball of center $x$ and radius $r$.) For many natural measures $\mu$, such as self-similar measures and measures invariant and ergodic under hyperbolic diffeomorphisms, the local dimension exists and is constant $\mu$-almost everywhere. However, even in this case there may be many points for which the local dimension takes exceptional values (or fails to exist). The multifractal analysis for local dimensions is broadly concerned with the study of the level sets
\[
\Delta_\mu(\alpha)=\{ x: d(\mu,x)=\alpha\}.
\]
The goal is to compute or estimate the \textit{dimension spectrum} of $\mu$, that is, the function
\[
f_\mu(\alpha) =  \dim_H(\Delta_\mu(\alpha)),
\]
where $\dim_H$ denotes Hausdorff dimension. Loosely speaking, a measure is termed \textit{multifractal} if $\dim_H(\Delta_\mu(\alpha))>0$ for a range of values of $\alpha$. For many natural measures, such as self-similar measures under the open set condition, it turns out that $f_\mu(\alpha)$ is the Legendre transform of the so-called $L^q$-spectrum $\tau_\mu(q)$, which is another quantity that reflects the global oscillations of $\mu$, and is often easier to compute. See e.g. \cite[Chapter 12]{Falconer97}.

A large literature is devoted to the study of the dimension spectrum of self-similar measures. In particular, in the case of self-similar measures satisfying the strong separation condition  the dimension spectrum was computed in \cite{CawleyMauldin} and this was extended to the self-similar measures satisfying the open set condition in \cite{ArbeiterPatzschke96}. However, despite substantial recent progress \cite{FengLauWang05, Feng07, FengLau09, Olsen09, Feng10}, the overlapping case remains rather mysterious. In particular, nearly all attention has been focused on \textit{singular} self-similar measures $\mu$. Very recently, Feng \cite[Proposition 5.1]{Feng10} has shown that certain absolutely continuous self-similar measures may also possess a rich multifractal structure. A main theme of this paper is to show that this is also the case for what is, perhaps, the simplest and most studied family of overlapping self-similar measures: (biased) Bernoulli convolutions. We review their definition and main properties.

\subsection{Bernoulli convolutions} \label{subsec:BCs}

For $\lam\in (0,1)$ and $p \in (0,1)$ let
$\nula^p$ be the $p$-biased Bernoulli convolution measure on the real line with contraction rate $\lam$, i.e., the self-similar measure satisfying the equation
\begin{equation}\label{eq:ss1}
\nula^p = p(\nula^p\circ S_0^{-1}) + (1-p)(\nula^p\circ S_1^{-1}),\ \ \mbox{where}\ S_j(x) = \lam (x + j), \ j=0,1.
\end{equation}
Equivalently, $\nula^p$ is the distribution of the random series $\sum_{n=1}^\infty i_n \lam^n$ where
$i_n$ is a sequence of i.i.d.\ random variables taking the values 0,1 with probabilities $p,1-p$. (In many papers the series starts with $n=0$; here it is more convenient
for us to start with $n=1$.)

Yet another useful characterization is the following: $\nu_\lam^p = \eta^p\circ\Pi_\lam^{-1}$, where $\eta^p$ is the Bernoulli measure $(p,1-p)^\N$ on $\{0,1\}^\N$, and
\[
\Pi_\lam((i_n)_{n=1}^\infty) = \sum_{n=1}^\infty i_n \lam^n = \lim_{n\to\infty} S_{i_1}\circ\cdots\circ S_{i_n}(0)
\]
is the ``natural projection'' from $\{0,1\}^{\N}$ to $\R$. In the unbiased, or symmetric, case $p=1/2$, we will write $\nula := \nula^{1/2}$.

We review some of the key facts about Bernoulli convolutions, with a view towards our investigations, while referring the reader to \cite{PeresSchlagSolomyak00, Solomyak04} and references therein for further background. Let
\[
s_p(\lam)= \frac{h(p)}{-\log(\lambda)},
\]
where $h(p) = -p\log(p)-(1-p)\log(1-p)$ is the $p$-entropy. Notice that $s_p(e^{-h(p)})=1$. As we will see, $s_p(\lam)$ is the ``typical'' or ``expected'' dimension of $\nu_\lam^p$. When $s_p(\lam)>1$, typically $\nu_\lam^p$ is absolutely continuous; when $s_p(\lam)<1$, on the other hand, $\nu_\lam^p$ is always singular:

\begin{theorem}\label{thm:absolute-continuity-BCs}{\em(Peres and Solomyak \cite{PeresSolomyak98})}

\textbf{(i)} For $p\in (\frac13,\frac23)$, the measure $\nu_\lam^p$ is absolutely continuous for almost every $\lam\in (e^{-h(p)},1)$. In particular, $\nu_\lam$ is absolutely continuous for almost every $\lam\in(\frac{1}{2},1)$.

\textbf{(ii)} On the other hand, $\dim_H(\nu_\lam^p)\le s_p(\lam)$ for all $\lam,p$, whence $\nu_\lam^p$ is singular for $\lam<e^{-h(p)}$. (Here $\dim_H(\cdot)$ denotes the Hausdorff dimension of a measure, defined as the infimum of the Hausdorff dimensions of sets of full measure.)
\end{theorem}

We make some further remarks:
\begin{remark}
Finer information on the regularity of the densities of $\nu_\lam^p$ is also available, see  \cite[Theorem 1.3]{PeresSolomyak98}, as well as the sharpening of these results obtained in \cite{PeresSchlag00}.  In the unbiased case $p=\frac12$ one can say rather more. For example, $\nu_\lam$ has a continuous density for almost every $\lam\in (2^{-1/2},1)$; this follows from Theorem \ref{thm:absolute-continuity-BCs} and a convolution argument, see \cite{Solomyak95}.
\end{remark}
\begin{remark}
 H. Toth \cite{Toth08} showed that for all $p\in (0,1)$, $\nula^p$ is absolutely continuous for almost every $\lam$ in a non-trivial interval $(\lam_p,1)$. However, for $p$ outside $(\frac13,\frac23)$ it is not known whether one can take $\lam_p = e^{-h(p)}$. (The interval $(\frac13,\frac23)$ can be expanded somewhat, but existing techniques break down for $p$ close to $0$ or $1$.)
\end{remark}

When $\nula^p$ is absolutely continuous, it is in fact equivalent to Lebesgue measure on its support:

\begin{theorem} \label{thm:mutual} {\em \cite{AlexanderParry88,MauldinSimon98}}
Let $I_\lam: = \supp(\nula)= [0,\frac{\lam}{1-\lam}]$.  If $\nula^p \ll \Leb$ then $\Leb|_{I_\lam}$ and $\nula^p$
are mutually absolutely continuous ($\Leb$ denotes Lebesgue measure on the line).
\end{theorem}

If $\lam^{-1}$ is a Pisot number, i.e. an algebraic integer larger than $1$ all of whose algebraic conjugates are smaller than $1$ in modulus, then $\nu_\lam^p$ is singular for all $p$. It is a major open question whether there are other values of $\lam\in(\frac{1}{2},1)$ for which $\nu_\lam^{1/2}$ is singular. The only other concrete class of numbers which are believed to, possibly, yield singular Bernoulli convolutions are reciprocal Salem numbers, see e.g. \cite{Feng10} for their definition and main properties. In particular, in \cite{Feng10} Feng proves that Bernoulli convolutions associated with (reciprocal) Salem numbers have a rich multifractal spectrum, a fact which (in the unbiased case) had been only established in the Pisot case.  In this paper we focus on results which are valid for all or typical $\lambda$, and hence do not deal with these matters.

%%%%%%%%%%%%%%%%%%%%%%%%%%%%%%%%%%%%%%%%%%%%%%%%%%%%%%%%%%%%%%%%%%%%%%%%%%%%%%

\subsection{Results: the unbiased case} \label{subsec:results-unbiased}

In this section we consider the symmetric case $p=\frac{1}{2}$. Recall that $\nula=\nula^{1/2}$; we also write $s(\lam):=s_{1/2}(\lam)=\log 2/|\log\lam|$. When $s(\lam)>1$, one may ask about further properties of the density $d\nula/dx$ (i.e. the Radon-Nikodym derivative) for a typical $\lam$.  It is not hard to see that  $d\nula/dx$ becomes $0$ at the two endpoints of the support of $\nula$ (in fact, the local dimension at these two points is $s(\lam)>1$). In \cite[Question 8.3.1]{PeresSchlagSolomyak00} it was asked whether there can be any other zeros of the density. It is easy to deduce from Theorem~\ref{thm:mutual} and the (typical) existence of continuous density
that for a.e. $\lam > 2^{-\frac{1}{2}}$ the density
$\frac{d\nula}{dx}\in C(\R)$ does not vanish in $\Int(I_\lam)$.
 On the other hand we show that for {\em all}
 $\lam \in (\half,g)$, where $g = \frac{\sqrt{5}-1}{2}$ is the (reciprocal) golden ratio, there exist infinitely many $x\in \Int(I_\lam)$ such that
\begin{equation} \label{eq:deriva}
D(\nula,x)=\lim_{r\searrow 0} \frac{\nula (B(x,r))}{2r}=0.
\end{equation}
In fact, much more is true, as we will see below. Recall the definition of the local dimension $d(\mu,x)$ given in \eqref{eq:def-local-dim}. One may consider also the lower and upper local dimensions $\underline{d}(\mu,x)$, $\overline{d}(\mu,x)$, by taking the lower and upper limits respectively. These quantities are always defined.

In the next theorem, $\beta_c$ is the Komornik-Loreti constant defined as the unique positive solution of the equation $1 = \sum_{n=1}^\infty {\mathfrak m}_n x^{-n+1}$,
where $({\mathfrak m}_n)_1^\infty$ is the Thue-Morse sequence $0110\ 1001\ 1001\ 0110\ \ldots$, see \cite{KomornikLoreti98}. We have $\beta_c^{-1} = 0.5598\ldots$
Recall that $g = \frac{\sqrt{5}-1}{2} = 0.618034\ldots$

\begin{theorem} \label{thm:class1}
For all $\lam \in (\frac{1}{2}, g)$ there exists a set $\wtil{\Ak}_\lam \subset I_\lam$ such that
\begin{equation} \label{eq:locdim1}
d(\nula,x) =s(\lam) >1\ \ \mbox{for all}\ \  x\in \wtil{\Ak}_\lam.
\end{equation}
Moreover,

{\bf (i)} $\wtil{\Ak}_\lam$ is countably infinite for $\lam \in (\beta_c^{-1},g)$;

{\bf (ii)} $\dim_H\wtil{\Ak}_\lam > 0$ for $\lam \in (\frac{1}{2},\beta_c^{-1})$;

{\bf (iii)} $\dim_H\wtil{\Ak}_\lam \to 1$ as $\lam \searrow \frac{1}{2}$.
\end{theorem}

\begin{remark}
The above fails beyond the golden ratio: Feng and Sidorov \cite[Corollary 1.6]{FengSidorov10} have shown that $\overline{d}(\nula,x)<s(\lam)$ for all $\lam\in (g,1)$ and $x$ in the interior of $I_\lam$.
\end{remark}
\begin{remark}
It is known that for all $\lambda\in\left(\frac{1}{2},1\right)$ there exists a constant $0<\delta(\lambda)=\dim_H\nu_{\lambda}\leq 1$ such that $d(\nu_{\lambda},x)=\delta(\lambda)$ for $\nu_{\lambda}$ almost all $x$. This was stated in \cite{Led} with an outline of a proof, and formally proved in \cite{FengHu09}. Thus by combining this result with Theorem \ref{thm:class1}, we can deduce that for all $\lambda\in \left(\frac{1}{2},\beta_c^{-1}\right)$ there are at least two distinct values of $\alpha$ with $\dim_H\Delta_{\nu_{\lambda}}(\alpha)>0$.
\end{remark}
The definition of the sets $\wtil{\Ak}_\lam$ is given below, in \S \ref{subsec:expansions}.

%%%%%%%%%%%%%%%%%%%%%%%%%%%%%%%%%%%%%%%%%%%%%%%%%%%%%%%%%%%%%%%%%%%%%%%%%%%%%%

\subsection{Results: the biased case $p \ne \frac{1}{2}$.} \label{subsec:results-biased}

Theorem \ref{thm:class1} provides information about a single exceptional value of the local dimension: the similarity dimension $s(\lam)$. Recall that $\nula^p$ is a non-linear projection (under the coding map) of $\eta^p$, the $p$-Bernoulli measure on the symbolic space $\{0,1\}^\N$. This measure is multifractal for $p\neq \frac{1}{2}$, which suggests that, in the asymmetric case, $\nula^p$ may inherit a degree of multifractality. The results of this section show that this is, indeed, the case.

We start with an extension of Theorem \ref{thm:class1}. We need another number here, $\beta_1$, which is the positive solution of the equation $ 1 = x^{-1} + \sum_{n=1}^\infty x^{-2n}$, we have $\beta_1^{-1} = 0.554958\ldots$
Assume that $p \in (0, \half)$ without loss of generality.

\begin{theorem} \label{thm:bias1}
Let $\wtil{\Ak}_\lam$ be the same sets
as in Theorem~\ref{thm:class1}. Then

{\bf (i)} For all $\lam \in (\beta_1^{-1},g)$,
\begin{equation} \label{eq:locdim2}
d(\nula^p,x) =\frac{\log p + \log(1-p)}{2\log\lam} \ \ \mbox{for all}\ \  x\in \wtil{\Ak}_\lam;
\end{equation}

{\bf (ii)} For all $\lam \in (\frac{1}{2}, \beta_1^{-1})$ there exists $r_\lam\in (0,\half)$ such that
\begin{eqnarray*}
& & \dim_H\{x\in \wtil{\Ak}_\lam:\ d(\nula^p,x) = \gam\} > 0\\[1.2ex] & & \mbox{for all}\ \ \gam \in
\left[\frac{r_\lam \log p + (1-r_\lam) \log(1-p)}{\log\lam}, \frac{(1-r_\lam) \log p + r_\lam \log(1-p)}{\log\lam} \right].
\end{eqnarray*}
Moreover, $r_\lam\searrow 0$ as $\lam\searrow \half$;

{\bf (iii)} For all $\lam \in (\frac{1}{2}, \beta_1^{-1})$,
\begin{equation} \label{eq:locdim4}
\dim_H\{x\in \wtil{\Ak}_\lam:\ \ \underline{d}(\nula^p,x) < \ov{d}(\nula^p,x)\} > 0.
\end{equation}
\end{theorem}
The proof will in fact yield quantitative estimates on $r_\lam$ and the Hausdorff dimensions of the sets involved. The second part of the theorem shows that, for $\lam  \in (\frac{1}{2}, \beta_1^{-1})$ and $p\neq \frac12$, the spectrum $\dim_H(\Delta_{\nula^p}(\alpha))$ is strictly positive for $\alpha$ in some interval which extends (and possibly starts) beyond $1$.

Our next theorem shows that also many \emph{small} local dimensions ($\alpha<1$) arise. Unlike Theorem \ref{thm:bias1}, this is an almost-everywhere result: it holds not only for typical $\lam$, but also for typical $\alpha$. Moreover, we are only able to obtain results for $\lam$ in a so-called \textit{interval of transversality}, i.e. for $\lam\in (0,\lam_*)$, where $\lam_*$ has the property that Proposition \ref{prop:transversality} below holds. In particular, we can take $\lam_*=0.66847$. See e.g. \cite{Solomyak04, ShmerkinSolomyak06} for a discussion of transversality in this context and how to find intervals of transversality. Notice that $\lam_* > g$.

Nevertheless, we are able to see that biased Bernoulli convolutions can have a rich spectrum of small dimensions, even in the parameter region on which they are typically absolutely continuous. We underline that it is often harder to obtain estimates for $\dim_H(\Delta_\mu(\alpha))$ when $\alpha$ is smaller than the Hausdorff dimension of $\mu$; see e.g. \cite{FengLauWang05, Feng07} for some instances of this. Again assume, without loss of generality, that $p\in(0,\frac{1}{2})$ is fixed.

\begin{theorem} \label{thm:small-dim-BCs}
For almost every $\lam\in(\frac{1}{2},\min(\lam_*,1-p))$, the following holds: the set $\Delta_{\nula^p}(\alpha)$ has positive Hausdorff dimension for almost every $\alpha$ in
\[
J(p,\lam) = \left[\frac{\log(1-p)}{\log\lam},1 \right].
\]
More precisely,
\[
\dim_H(\Delta_{\nula^p}(\alpha)) \ge \frac{h(q)}{|\log\lam|} \quad \text{for almost every } \alpha\in J(p,\lam),
\]
where $\alpha$ and $q$ are related by
\begin{equation} \label{eq:relation-alpha-q}
\log\lam\cdot\alpha=q\log p+(1-q)\log(1-p).
\end{equation}
\end{theorem}

\begin{remark}
Notice that, for any $p\in (\frac{1}{3},\frac{1}{2})$, there is an open set of $\lam$ for which $\nu_\lam^p$ is (typically) absolutely continuous and the interval $J(p,\lam)$ in the above theorem is nontrivial. Indeed, by Theorem \ref{thm:absolute-continuity-BCs}, $\nula^p$ is absolutely continuous for a.e. $\lam\in (p^p(1-p)^{1-p},1)$, while $J(p,\lam)$ is nontrivial if $\lam<\min(\lambda_*,1-p)$. Since $p^p(1-p)^{1-p}<1-p$ precisely for $p<\frac{1}{2}$, we see that, even for $p$ very close to $\frac{1}{2}$, there are many values of $\lam$ (more precisely, almost every value in some interval) for which $\nula^p$ is absolutely continuous, yet has a positive measure set of local dimensions smaller than $1$.
\end{remark}

Our last result is of a slightly different kind: it concerns the behavior of the multifractal spectrum as $\lam\to \frac{1}{2}$. We note that, a priori, there is no reason why, for a fixed $x$, $d(\nula^p,x)$ should be continuous in $\lam$. One may ask whether, in spite of this, the \textit{global} spectrum $\dim_H(\Delta_{\nula^p}(\alpha))$ behaves continuously with the parameters $\lambda$ and $p$. In general this appears unlikely, since e.g. reciprocals of Pisot numbers are exceptional. However, the next theorem shows that, as the size of the overlaps tends to $0$, there is continuity:

\begin{theorem}\label{thm:continuity-mf-spectrum}
Let
\[
f_{\lambda,p}(\alpha) = f_{\nula^p}(\alpha) = \dim_H\left(\Delta_{\nula^p}(\alpha)\right).
\]
Then, for all $p_0\in (0,1)$ and all $\alpha\ge 0$,
\[
\lim_{\lam\to \frac{1}{2}, p\to p_0} f_{\lambda,p}(\alpha) = f_{1/2,p_0}(\alpha).
\]
\end{theorem}

Although we do not make them explicit, the proof of the theorem provides computable lower and upper bounds: see Figure \ref{fig:mfbounds}.
\begin{figure}
\includegraphics[width=0.9\textwidth]{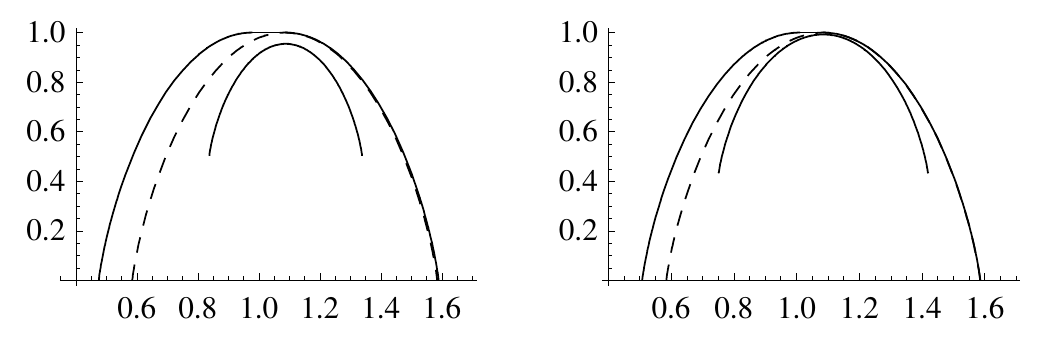}
\caption{The graph shows explicit upper and lower bounds for $\alpha\to f_{\lam,1/3}(\alpha)$, for $\lam=0.501$ (left) and $\lam=0.5001$ (right). The dashed curve is $f_{1/2,1/3}(\alpha)$ (which can be computed explicitly) in both cases. The upper and lower bounds are obtained from the proof of Theorem \ref{thm:continuity-mf-spectrum}.} \label{fig:mfbounds}
\end{figure}

\subsection{Notation}

The following table summarizes the main notation to be used throughout the paper.

\medskip

\begin{center}
\begin{tabular}{|c|c|}
\hline
Object & Notation \\
\hline
Cardinality of a set $A$ & $\card A$ \\
Length of an interval $I$ & $|I|$ \\
Finite or infinite string of $0$s and $1$s & $\iii,\jjj,\kkk$ \\
$n$-th element of $\iii,\jjj,\kkk$ & $i_n,j_n,k_n$\\
Concatenation of $\iii$ and $\jjj$ & $\iii\jjj$ \\
Empty word & $\varnothing$ \\
Length of a finite word $\iii$ & $|\iii|$ \\
Restriction of $\iii$ to its first $n$ elements & $\iii|n$ \\
Longest common initial subword of $\iii$ and $\jjj$ & $\iii\wedge\jjj$\\
All infinite words starting with a finite string $\iii$ & $[\iii]$ \\
String of $k$ consecutive $0$'s (resp. $1$'s) & $0^k$ (resp. $1^k$)\\
Left shift operator on $\{0,1\}^\N$ & $\sigma$\\
\hline
\end{tabular}
\end{center}
\subsection*{Acknowledgement} We would like to thank Mark Pollicott and Nikita Sidorov for useful discussions.
\medskip

%%%%%%%%%%%%%%%%%%%%%%%%%%%%%%%%%%%%%%%%%%%%%%%%%%%%%%%%%%%%%%%%%%%%%%%%%%%%%%
%%%%%%%%%%%%%%%%%%%%%%%%%%%%%%%%%%%%%%%%%%%%%%%%%%%%%%%%%%%%%%%%%%%%%%%%%%%%%%

\section{Proof of Theorems~\ref{thm:class1} and \ref{thm:bias1}} \label{sec:BCs-proofs}

\subsection{Expansions in base $\lam$.} \label{subsec:expansions}
Recall that $I_\lam = \supp(\nula) = [0,\frac{\lam}{1-\lam}]$. For $x\in I_\lam$ define
$$
\Ek_\lam(x) = \Bigl\{(a_1,a_2,\ldots)\in \{0,1\}^\Nat:\ x = \sum_{n=1}^\infty
a_n \lam^{n}\Bigr\}
$$
to be the set of all expansions of $x$ in base $\lam$ with digits 0 and 1.
Let
$$\Ak_\lam:=\{x\in I_\lam:\ \card\Ek_\lam(x)=1\}$$
be the set of $x$ having a unique expansion in base $\lam$.
Of course, 0 and $\frac{\lam}{1-\lam}$ always have a unique
expansion.

The following results are known about $\Ak_\lam$:

\begin{itemize}

\item For all $\lam \in [g,1)$ and all $x\in \Int(I_\lam)$, the set
$\Ek_\lam(x)$ is uncountable, of positive Hausdorff dimension.
On the other hand, for all $\lam \in (\half,g)$, the set $\Ak_\lam$ is
infinite \cite{ErdosJooKomornik90}.

\item For all $\lam\in (\beta_c^{-1},g)$, the set $\Ak_\lam$ is countably infinite;
for $\lam=\beta_c^{-1}$, the set $\Ak_\lam$ is uncountable of zero Hausdorff
dimension; for all $\lam \in (\half,\beta_c^{-1})$ the set $\Ak_\lam$
has positive Hausdorff dimension \cite{GlendinningSidorov01}.

\end{itemize}
We need some basic facts about $\beta$-expansions (see e.g.\ \cite{DajaniKraaikamp02}),
as well as some more recent results \cite{GlendinningSidorov01}. Let $\beta > 1$.
We use the notation $x \sim (a_1 a_2 a_3\ldots)_\beta$
%or  $\beta\,:\,x \sim (a_n)_{n\ge 1}$
to indicate $x = \sum_{n=1}^\infty a_n \beta^{-n}$. We will have
$\beta = \lam^{-1} \in (1,2)$, so the digits $a_n$ will always be in $\{0,1\}$. Given $x\in [0,1]$, the {\em greedy expansion} of $x$ in base $\beta = \lam^{-1}$ is
defined as the greatest sequence in $\Ek_\lam(x)$ in the lexicographic order $\prec$ on $\{0,1\}^\N$.
Alternatively, the greedy expansion is given by the symbolic dynamics of the ``greedy'' $\beta$-transformation
$$
G_\beta(x) = \left\{ \begin{array}{ll} \beta x, & x\in [0, \lam) \\ \beta  x -1, & x\in [\lam,\frac{\lam}{1-\lam}]\end{array} \right..
$$
Namely, the digit $a_n$ of the greedy expansion is $0$ if $G_\beta^{n-1}(x) \in [0,\lam)$, and 1 otherwise.
The {\em lazy expansion} of $x$ in base $\beta = \lam^{-1}$ is the smallest sequence in $\Ek_\lam(x)$ in the lexicographic order. Alternatively, it is given by the
symbolic dynamics of the ``lazy'' $\beta$-transformation
$$
L_\beta(x) = \left\{ \begin{array}{ll} \beta x, & x\in [0, \frac{\lam^2}{1-\lam}] \\ \beta  x -1, & x\in (\frac{\lam^2}{1-\lam},\frac{\lam}{1-\lam}]
\end{array} \right.,
$$
that is, the $n$-th digit is $0$ if $L_\beta^{n-1}(x) \in [0,\frac{\lam^2}{1-\lam}]$, and 1 otherwise.
Note that the functions $G_\beta$ and $L_\beta$ agree on $I_\lam \setminus C_\lam$, where $C_\lam=[\lam, \frac{\lam^2}{1-\lam}]$; see Figure \ref{fig:greedylazy}. A number $x\in [0,1]$ has a unique expansion in base $\beta$ if
and only if its greedy expansion coincides with its lazy expansion.

%\begin{figure}
% \centering
%  \includegraphics[width=0.75\textwidth]{greedylazy.eps}
%\caption{The greedy (left) and lazy (right) $\beta$-transformations. They agree outside a central segment, emphasized in the figure.} %\label{fig:greedylazy}
%\end{figure}

Let $F_\beta:\, I_\lam \setminus C_\lam \to I_\lam$ be the function equal to $G_\beta$ and $L_\beta$ on
the set of agreement. In other words,
$$
F_\beta(x) = \left\{\begin{array}{ll} \beta x, & x\in [0,\lam) \\
          \beta x-1, & x\in (\frac{\lam^2}{1-\lam},\frac{\lam}{1-\lam}] \end{array}\right..
$$

Note that there is a ``gap'' $C_\lam$ in the domain of $F_\beta$, so not every orbit is well-defined. In fact, the infinite orbit
$\{F_\beta^n x\}_{n\ge 0}$ exists if and only if $x$ has a unique expansion in base $\lam$.

\begin{figure}[h]
\centering
\begin{picture}(200,125)
\put(0,25){\vector(1,0){150}}
\put(25,0){\vector(0,1){125}}
\thicklines
\put(25,25){\line(2,3){25}}
\put(50,62.5){\line(2,3){25}}
\put(50,25){\line(2,3){25}}
\thinlines
\put(100,25){\line(0,1){75}}
\put(25,100){\line(1,0){75}}
\thicklines
\put(75,62.5){\line(2,3){25}}
\put(50,25){\dashbox{5}(25,75)[b]{$^{C_\lam}$}}
\put(14,100){\makebox(0,0){$\frac{\lam}{1-\lam}$}}
\put(50,13){\makebox(0,0)[b]{$\lam$}}
\put(75,7){\makebox(0,0)[b]{$\frac{\lam^2}{1-\lam}$}}
\put(100,7){\makebox(0,0)[b]{$\frac{\lam}{1-\lam}$}}
\end{picture}
\caption{The function $F_\beta$ is defined on the complement of the ``overlap region'' $C_\lam$. If we extend $F_\beta$
to $C_\lam$ using the ``lower branch'' in the dashbox, we get the greedy $\beta$-transformation, and taking the
``upper branch'' results in the lazy $\beta$-transformation.} \label{fig:greedylazy}
\end{figure}
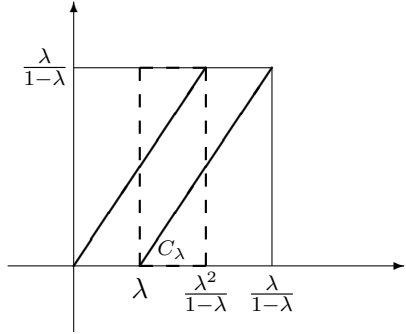

Now we can define the sets which appear in Theorems \ref{thm:class1} and \ref{thm:bias1}. Let
\begin{equation} \label{eq:definition-A-lam-tilde}
\wtil{\Ak}_\lam:=\bigcup_{\delta>0} \wtil{\Ak}_{\lam,\delta}:=
\bigcup_{\delta>0} \left\{x\in \Ak_\lam:\ \dist(\{F_\beta^nx\}_{n\ge 0}, C_\lam) \ge \delta \right\}\,,
\end{equation}
where $\beta=\lam^{-1}$. The statements in Theorem \ref{thm:class1} about the size of $\wtil{\Ak}_\lam$  follow from \cite{GlendinningSidorov01}, as we now explain.

%Let $\sig$ denote the left shift on $\{0,1\}^\N$.
%By definition, $\nula^p = (p,1-p)^\N \circ \Pi_\lam^{-1}$.
Let
$$
\Uk_\lam = \Pi_\lam^{-1}(\Ak_\lam),\ \ \ \ \ \wtil{\Uk}_\lam = \Pi_\lam^{-1}(\wtil{\Ak}_\lam).
$$
Observe that for $\iii \in \Uk_\lam$ and $x = \Pi_\lam(\iii) \in \Ak_\lam$ we have, for $\beta= \lam^{-1}$,
$$
F_\beta^n(x) = \Pi_\lam(\sig^n \iii).
$$
Also note that the digits $i_n$ of the expansion of $x\in \Ak_\lam$ are given by the symbolic dynamics of $F_\beta$.

W. Parry \cite{Parry60} characterized the set of all sequences arising from greedy $\beta$-expansions as
those which are lexicographically less than the greedy expansion of 1, and the same holds for all their shifts (a minor modification
is necessary if the greedy expansion of 1 is finite).

The next lemma follows from the characterization of unique expansions as being both greedy and lazy.
Denote by the bar the ``flip'' of ``reflection'', i.e., $\ov{1}=0,\ \ov{0}=1$.

\begin{lemma} \label{lem2} {\em \cite[Lemma 4]{GlendinningSidorov01}} Let $\lam \in (\half,1)$ and
let $1 = \sum_{n=1}^\infty d_n \lam^n$ be the greedy expansion of 1 in base $\beta = \lam^{-1}$.
Then
\begin{eqnarray*}
\Uk_\lam = \{ \iii \in \{0,1\}^\N:\ (i_n,i_{n+1},\ldots ) & \prec & (d_1,d_2,\ldots) \\
\mbox{\rm and}\ \ (\ov{i}_n, \ov{i}_{n+1},\ldots) & \prec & (d_1,d_2,\ldots),\ n\in \N\}.
\end{eqnarray*}
\end{lemma}

\begin{lemma} \label{lem:nested-unique-exp}
For $\half < \lam_1 < \lam_2< g$, we have
$$
\Uk_{\lam_1}\supset \wtil{\Uk}_{\lam_1} \supset \Uk_{\lam_2} \supset \wtil{\Uk}_{\lam_2}.
$$
\end{lemma}
\begin{proof}
We only need to show the middle inclusion. Let $\iii \in \Uk_{\lam_2}$. Then
\begin{eqnarray*}
\Pi_{\lam_2}(\sig^{n-1} \iii) < \lam_2,\ & \mbox{if} \ i_{n} = 0,\\
\Pi_{\lam_2}(\sig^{n-1} \iii) > \frac{\lam_2^2}{1-\lam_2},\  & \mbox{if} \ i_{n} = 1.
\end{eqnarray*}
First assume that $i_{n} = 0$.
If $i_{n+1}=0$, then
$$
\Pi_{\lam_1}(\sig^{n-1}\iii)  = \sum_{k=n}^\infty i_k \lam_1^{k-n+1} \le \frac{\lam_1^3}{1-\lam_1} < \lam_1,
$$
since $1-\lam_1 - \lam_1^2 > 0$ for $\lam_1 < g$.
If $i_{n+1}=1$, then
\begin{eqnarray*}
\Pi_{\lam_1}(\sig^{n-1} \iii)  & = & \lam_1^2 + \sum_{k=n+2}^\infty i_k \lam_1^{k-n+1} \\
                              & \le & (\lam_1^2-\lam_2^2) + \Pi_{\lam_2}(\sig^{n-1}\iii) \\
                              & < & (\lam_1^2-\lam_2^2) + \lam_2 < \lam_1,
\end{eqnarray*}
using that $\lam_2> \lam_1$ and $\lam_1 + \lam_2 > 1$.

The case $i_n = 1$ reduces to the previous one by symmetry, by considering $\ov{\iii}=(\ov{i}_1,\ov{i}_2,\ldots)$.
Thus, $\Uk_{\lam_2} \subset
\Pi_{\lam_1}^{-1}(\wtil{\Ak}_{\lam_1,\delta})\subset \wtil{\Uk}_{\lam_1}$, where
$$
\delta = \min\left\{(\lam_2-\lam_1)(\lam_1+\lam_2-1),\, \frac{\lam_1(1-\lam_1-\lam_1^2)}{1-\lam_1}\right\}.
$$
\end{proof}

\medskip

Let
$$
\ell_j(\iii,n) = \card\{k\in \{1,\ldots,n\}:\ i_k = j\}\ \mbox{for}\ j=0,1,
$$
and define the frequency of $j$'s in $\iii$ by
$$
\freq_j(\iii) = \lim_{n\to \infty} n^{-1} \ell_j(\iii,n),
$$
if the limit exists.
In the next lemma, we equip the sequence space $\{0,1\}^\N$ with the metric $\varrho(\iii,\jjj) = \lam^{|\iii\wedge \jjj|}$. The Hausdorff dimension on $\{0,1\}^\N$ is calculated with respect to this metric. (Of course, one may change the power of the exponent from $\lam$ to any $\gam \in (0,1)$; the base $\half$ is a common choice.)

\begin{lemma} \label{lem:freq} {\bf (i)} For all $\lam \in [\beta_1^{-1},g)$,
$$
\freq_0(\iii) =\frac{1}{2}\ \ \mbox{for all} \ \iii \in \Uk_\lam.
$$

{\bf (ii)} For all $\lam \in (\half,\beta_1^{-1})$ there exists $r_\lam \in (0,\half)$ such that
\begin{equation} \label{eq:dim2}
\dim_H\{\iii \in \Uk_\lam:\ \freq_0(\iii) =r\} > 0\ \ \mbox{for all}\ r\in (r_\lam, \,1-r_\lam).
\end{equation}
Moreover, $r_\lam\searrow 0$ as $\lam\searrow \half$. Also,
$$
\dim_H\{\iii \in \Uk_\lam:\ \freq_0(\iii)\ \mbox{does not exist}\}>0.
$$
Furthermore,
\begin{equation} \label{eq:dim-tends-to-1}
\dim_H(\Uk_\lam) \to 1 \quad\text{as } \lam\searrow \frac12.
\end{equation}
\end{lemma}
\begin{remark}
Our proof uses techniques from \cite{GlendinningSidorov01}. The fact \eqref{eq:dim-tends-to-1} is stated in \cite[Example 17]{GlendinningSidorov01} with a sketch of proof; we include a more detailed proof for completeness.
\end{remark}
\begin{proof}[Proof of Lemma \ref{lem:freq}]
(i)
In view of the monotonicity of $\Uk_\lam$, it is enough to prove the statement for $\lam = \beta_1^{-1}$. The greedy expansion of 1 in base $\beta_1$ is
$1 \sim 1(10)^\infty = (1101010\ldots)_{\beta_1}$. If $\iii \in \Uk_\lam$, then by Lemma~\ref{lem2},
$$
(0010101010\ldots) \prec (i_n, i_{n+1} \ldots ) \prec (11010101010\ldots)
$$
It follows that $\iii$ is an arbitrary concatenation of the words in
$$
\Pk:= \{10, 1100, 110100,\ldots,1(10)^k 0,\ldots\}
$$
and their flips. This implies that for all $n$,
$$
|\ell_0(\iii,n) - (n/2)| \le 1,
$$
and the claim follows.

(ii)
If $\lam< \beta_1^{-1}$, then $\beta = \lam^{-1} > \beta_1$ and hence the greedy $\beta$-expansion
$1=\sum_{n=1}^\infty d_n\beta^{-n}$ satisfies $1(10)^\infty \prec (d_n)_1^\infty$. Thus, there exists $k\ge 0$
such that
$$
d_1\ldots d_{2k+3} = 1(10)^k11.
$$
Then $\Uk_\lam$ contains arbitrary concatenations of words $u_0:=1(10)^{k+1}$ and $u_1:= 0(01)^{k+1}$, which can be written as $\prod_1^\infty \{u_0,u_1\}$.
This clearly allows one to get any frequency of 0's in $\iii$ in the
interval $[\frac{k+1}{2k+3}, \frac{k+2}{2k+3}]$, as well as having no frequency at all, taking $\iii =  u_{n_1} u_{n_2} \ldots$ with $n_1 n_2\ldots \in \{0,1\}^\N$.
Moreover, since the set of arbitrary 0-1 sequences $(i_n)_{n\ge 1}$ with a given frequency of 0's in $(0,1)$, has positive dimension, (\ref{eq:dim2}) follows.
Also, the set of arbitrary 0-1 sequences $(i_n)_{n\ge 1}$ with no frequency of 0's, has positive (full) dimension.

It remains to show that $r_\lam\searrow 0$ and $\dim_H(\Uk_\lam)/s(\lam)\to 1$ as $\lam\searrow \half$. To this end, consider the sequence $g_k$ of ``multinacci numbers'',
i.e.\ $1=\sum_{i=1}^k g_k^i$,
so that $g= g_2$. Clearly, $g_k\searrow \half$ as $k\to \infty$. Let $\lam \in (\half,g_k)$. Then any sequence in $\{0,1\}^\N$ without $0^k$ and $1^k$ belongs to
$\Uk_\lam$. In particular, $\Uk_\lam$ contains arbitrary concatenations of $v_0 = 0^{k-1}1$ and $v_1=01^{k-1}$. Taking sequences $\iii=v_{n_1}v_{n_2}\ldots$
we can achieve any frequency of 0's in $\iii$ in the interval $(\frac{1}{k},\frac{k-1}{k})$ on a set of positive dimension.

Also, $\Uk_\lam$ contains all sequences of the form $10\jjj_1 10\jjj_2\ldots$, where $\jjj_1,\jjj_2,\ldots \in \{0,1\}^{k-2}$ are arbitrary. A standard calculation yields
\[
\dim_H(\Uk_\lam) \ge \frac{\log(2^{k-2})}{|\log\lambda^k|} = \frac{k-2}{k} s(\lam).
\]
Since $\dim_H(\{0,1\}^\N)=s(\lam)\to 1$ as $\lam\searrow\half$, the proof is finished.
\end{proof}

%%%%%%%%%%%%%%%%%%%%%%%%%%%%%%%%%%%%%%%%%%%%%%%%%%%%%%%%%%%%%%%%%%%%%%%%%%%%%%

\subsection{Proof of the theorems} \label{subsect:proof-of-expansion-theorems}

We start with a lemma which says that, for $x\in \wtil{\Ak}_{\lam,\delta}$, the measure of of a ball centered at $x$ is comparable to the measure of the corresponding symbolic cylinder.

\begin{lemma} \label{lem:meas-when-unique-exp}
For any $\delta>0$ there exists $c_\delta>0$ such that the following holds:
if $x\in \wtil{\Ak}_{\lam,\delta}$ and $\iii\in \{0,1\}^{\N}$ is the unique sequence satisfying $\Pi_\lam(\iii)=x$, then
\begin{equation} \label{eq:ineq1}
\forall\,n\ge 1,\ \ c_\delta p^{\ell_0(\iii,n)} (1-p)^{\ell_1(\iii,n)} \le \nula^p(B(x,\delta\lam^n)) \le
p^{\ell_0(\iii,n)} (1-p)^{\ell_1(\iii,n)}.
\end{equation}
\end{lemma}
\begin{proof}
Suppose that $x\in I_\lam \setminus C_\lam$, and moreover, $r\le \dist(x,C_\lam)$. Denote $p_0=p, p_1 = 1-p$, and let
$j=i_1$. It follows from (\ref{eq:ss1}) that
\begin{equation} \label{eq:iter}
\nula^p(B(x,r)) = p_j \nula^p(S_j^{-1}B(x,r)) = p_j \nula^p(B(S_j^{-1} x,\lam^{-1} r)).
\end{equation}
We can apply (\ref{eq:iter}) $n$ times for the ball $B(x,\delta \lam^n)$ to obtain
$$
\nula^p(B(x,\delta\lam^n)) = \prod_{k=0}^{n-1} p_{i_k} \nula^p(B(y, \delta))=
p^{\ell_0(\iii,n)} (1-p)^{\ell_1(\iii,n)} \nula^p(B(y,\delta)),
$$
where $y = S_{i_1}^{-1}\cdots S_{i_n}^{-1} x$. Now the upper bound in (\ref{eq:ineq1}) is immediate, and the lower bound follows from the fact that
$$
\inf_{x\in I_\lam}\nula^p(B(x,\delta)) =:c_\delta>0.
$$
This is easy to see by self-similarity, again applying (\ref{eq:ss1}): if $N\log(1/\lam) > \log(\frac{1}{\delta(1-\lam)})$, then applying
 (\ref{eq:ss1}) $N$ times we obtain one of the sets in the right-hand side containing the entire support of $\nula^p$, so
$$
\nula^p(B(x,\delta)) \ge (\min\{p,1-p\})^N.
$$
\end{proof}

\medskip

Now the following is immediate from the last lemma and the definitions.

\begin{cor} \label{cor-cor3} For $x\in \wtil{\Ak}_\lam$ and $\iii \in \wtil{\Uk}_\lam$ with $\Pi_\lam(\iii) = x$ we have

{\bf (i)} if $r = \freq_0(\iii)$ exists, then
$$
d(\nula^p,x) = \frac{r\log p + (1-r) \log(1-p)}{\log\lam}\,;
$$

{\bf (ii)} if $\freq_0(\iii)$ does not exist, then $\underline{d}(\nula^p,x)< \ov{d}(\nula^p,x)$.
\end{cor}

Finally, the next lemma will allow us to transfer the dimension results of Lemma \ref{lem:freq} to the Euclidean setting.

\begin{lemma} \label{lem:bi-Lipschitz}
The map $\Pi_\lam|_{\Uk_\lam}$ is bi-Lipschitz from the metric $\varrho$ (defined before Lemma~\ref{lem:freq}) to the Euclidean metric on $\Ak_\lam$.
\end{lemma}
\begin{proof}
It is standard (and very easy) that $\Pi_\lam$ is Lipschitz on $\{0,1\}^\N$. For the other direction, let $\iii,\jjj\in \Uk_\lam$ and $x=\Pi_\lam(\iii),\ y = \Pi_\lam(\jjj)$.
Suppose $|\iii\wedge \jjj|=n$. Then
$i_{n+1}$ and $j_{n+1}$ are different, hence $F_\beta^n(x)$ and $F_\beta^n(y)$ are in different subintervals of $I_\lam\setminus C_\lam$. Thus,
$$
|\Pi_\lam(\iii) - \Pi_\lam(\jjj)| = \lam^n |\Pi_\lam(\sig^n\iii) - \Pi_\lam(\sig^n \jjj)| = \lam^n |F_\beta^n(x)-F_\beta^n(y)| \ge |C_\lam| \varrho(\iii,\jjj),
$$
and the lemma follows.
\end{proof}

\medskip

\begin{proof}[Proof of Theorem~\ref{thm:class1}] The main claim is immediate from Corollary~\ref{cor-cor3} taking $p=\half$.
The statements (i) and (ii) about $\wtil{\Ak}_\lam$ follow from the known results about $\Ak_\lam$ \cite{ErdosJooKomornik90,GlendinningSidorov01} mentioned above, together with Lemma~\ref{lem:nested-unique-exp}. Finally,  (iii) is a consequence of Lemmas \ref{lem:nested-unique-exp}, \ref{lem:freq}(ii) and \ref{lem:bi-Lipschitz} (since bi-Lipschitz maps preserve Hausdorff dimension).
\end{proof}

\medskip

\begin{proof}[Proof of Theorem~\ref{thm:bias1}] The statement (i) follows from Lemmas \ref{lem:freq}(i) and \ref{lem:meas-when-unique-exp}.

The statements (ii) and (iii) of the theorem follow from Lemma~\ref{lem:freq}(ii), Corollary~\ref{cor-cor3} and Lemma~\ref{lem:bi-Lipschitz}, using again that bi-Lipschitz maps preserve Hausdorff dimension.
\end{proof}

%%%%%%%%%%%%%%%%%%%%%%%%%%%%%%%%%%%%%%%%%%%%%%%%%%%%%%%%%%%%%%%%%%%%%%%%%%%%%%

\section{Proof of Theorem \ref{thm:small-dim-BCs}} \label{sec:proof-of-small-dim-BCs}

The proof of Theorem \ref{thm:small-dim-BCs} will be based on a combination of the potential-theoretic method with transversality arguments. To the best of our knowledge, this is the first instance in which transversality ideas are used to estimate the multifractal spectrum for a.e. parameter.

Theorem \ref{thm:small-dim-BCs} will be an easy consequence of the following stronger technical result:

\begin{theorem} \label{thm:small-dim-BCs-intermediate}
Fix $p,q\in (0,1)$. For almost all $\lam\in(0,\min\{\lambda_*,p^q(1-p)^{1-q}\})$,
\[
d(\nula^p,x) =\frac{q\log p + (1-q)\log(1-p)}{\log \lam} \quad\text{for }\nula^q\text{-a.e. } x.
\]
If $p^q(1-p)^{(1-q)}<\lambda^*$, then for almost all $\lam\in (p^q(1-p)^{(1-q)},\lam_*)$ we have
\[
d(\nula^p,x) \geq 1 \quad\text{for }\nula^q\text{-a.e. } x.
\]
\end{theorem}

We first indicate how to complete the proof of Theorem \ref{thm:small-dim-BCs}; the rest of this section is devoted to the proof of Theorem \ref{thm:small-dim-BCs-intermediate}.

\begin{proof}[Proof of Theorem \ref{thm:small-dim-BCs} (Assuming Theorem \ref{thm:small-dim-BCs-intermediate})]
Fix $p\in (0,\frac12)$. By Fubini, for almost every $\lam\in (\frac{1}{2},\lambda^*)$, the following holds for almost every $q\in (0,1)$ such that $\lam < p^q(1-p)^{1-q}$:
\[
d(\nula^p,x) = \frac{q\log p + (1-q)\log(1-p)}{\log \lam} \quad \text{for }\nula^q\text{-a.e. } x.
\]
It follows from Theorem \ref{thm:small-dim-BCs-intermediate} applied to $p=q$ that $\nula^q$ has Hausdorff dimension $h(q)/|\log(\lam)|$ for almost every $\lam\in (0,\lambda_*)$.  Let $\alpha$ be given by the relation \eqref{eq:relation-alpha-q}, and observe that $\alpha$ ranges between $\log (1-p)/\log\lam$ and $1$ (the upper bound is due to the restriction $\lam < p^q(1-p)^{1-q}$).

Let
\begin{equation} \label{eq:def-H}
H_p^q = - q\log(p) - (1-q)\log(1-p) > 0.
\end{equation}
Applying Fubini again, we obtain from the above that, for almost every $\lam\in (0,\lambda_*)$ such that $H_p^q<|\log\lam|$,
the set $\Delta_{\nula^p}(\alpha)$ has Hausdorff dimension at least $\dim_H(\nula^q) = h(q)/|\log(\lam)|$. This concludes the proof.
\end{proof}

We now start the proof of Theorem \ref{thm:small-dim-BCs-intermediate}. The upper bound is standard and holds for all $p,q,\lam$:

\begin{lemma} \label{lem:upper-bound-relative-BC}
With $H_p^q$ as in \eqref{eq:def-H},
\[
\overline{d}(\nu_\lam^p,x) \le \frac{H_p^q}{-\log\lam}\quad\text{for }\nula^q\text{-a.e. }x.
\]
\end{lemma}
\begin{proof}
Clearly,
\begin{equation} \label{eq:cylinder-inside-ball}
\Pi_\lam[\iii|n] \subset  B\left(\Pi_\lam(\iii), \diam(\Pi_\lam[\iii|n])\right) = B\left(\Pi_\lam(\iii), \frac{\lam}{1-\lam}\lam^n\right).
\end{equation}
Hence
\[
\overline{d}(\nu_\lam^p, \Pi_\lam(\iii)) \le \limsup_{n\to\infty} \frac{\eta^p([\iii|n])}{-\log\lam}.
\]
By the law of large numbers, the right-hand side equals $H_p^q/|\log\lam|$ for $\eta^q$-a.e. $\iii$, and this implies the lemma.
\end{proof}

In the remainder of this section we will find a lower bound for the $\nu_\lam^q$-typical value of $d(\nu_\lam^p,x)$, by employing transversality techniques. Unfortunately, as indicated earlier, we obtain results only for typical $\lam$ (for a fixed pair $p,q$).

The following is a simple but key lemma, which enables the use of the potential-theoretic method in the calculation of the multifractal spectrum.

\begin{lemma} \label{lem:potential-theoretic}
Let $\mu$ and $\nu$ be Borel probability measures on $\R^d$. Suppose that
\[
\int\int\frac{d\mu(x)d\nu(y)}{|x-y|^s}<\infty
\]
for some $s\ge 0$. Then $\underline{d}(\mu,x)\geq s$ for $\nu$-almost all $x$.
\end{lemma}
\begin{proof}
This is immediate from the fact that
\[
\underline{d}(\mu,x) = \sup\left\{t: \int |x-y|^{-t} \mu(y) <\infty\right\},
\]
which is well-known and easy to verify.
\end{proof}

In order to apply the previous Lemma to Bernoulli convolutions, we recall the following crucial transversality result; see \cite[Corollary 2.9]{ShmerkinSolomyak06}, or \cite{PeresSolomyak96} for an alternative but explicit approach that yields a somewhat worse constant $\lambda_*$.

\begin{prop} \label{prop:transversality}
Let $\lam_*=0.66847$. There exists a constant $C>0$ such that the following holds. Let $\iii,\jjj\in\{0,1\}^\N$ such that $i_1\neq j_1$. Then
\[
\Leb\{\lambda\in (0,\lam_*): |\Pi_\lambda(\iii)-\Pi_\lambda(\jjj)|<\delta\} < C\delta
\]
for all $\delta>0$.
\end{prop}

The following lemma is standard.

\begin{lemma} \label{lem:application-trans}
There exists $C>0$ such that for $\iii,\jjj\in\{0,1\}^{\mathbb{N}}$, $\lambda_0\in(\frac{1}{2},\lambda_*)$ and $s<1$,
\[
\int_{\lambda_0}^{\lambda_*} \frac{d\lambda}{|\Pi_{\lambda}(\iii)-\Pi_{\lambda}(\jjj)|^s}\leq C\lambda_0^{-s|\iii\wedge\jjj|}.
\]
\end{lemma}
\begin{proof}
We may assume that $i_1\neq j_1$. But then the lemma follows easily from Proposition \ref{prop:transversality} and Fubini.
\end{proof}

We are now able to conclude the proof of the theorem.

\begin{proof}[Proof of Theorem \ref{thm:small-dim-BCs-intermediate}]
Fix $p,q\in (0,1)$. In light of Lemma \ref{lem:upper-bound-relative-BC}, we only need to prove that, for almost every $\lam\in (0,\lambda_*)$,
\[
\underline{d}(\nu_\lam^p,x) \ge \min\left(\frac{H_p^q}{-\log\lam},1\right) \quad \text{for }\nu_\lam^q\text{-a.e. } x.
\]

Fix $\delta,\e>0$ and $\lam_0\in (0,\lam_*)$. Let
\[
s= \min\left(\frac{H_p^q-2\e}{-\log\lam_0},1-\e\right),
\]
and let $\Sigma\subset\{0,1\}^\N$ be a set of $\eta^q$-measure $1-\delta$, on which
\[
\frac{-\log\eta^p[\iii|n]}{n} \rightarrow H_p^q \quad\text{uniformly}.
\]
By uniform convergence, there exists a constant $C'=C'(\e)$ such that
\begin{equation} \label{eq:appl-unif-convergence}
\max_{ \iii\in\{0,1\}^n: \Sigma\cap [\iii]\neq \varnothing } \eta^p[\iii] \le C' e^{-n(H_p^q-\e)}.
\end{equation}
Since $\delta, \e$ and $\lam_0$ are arbitrary, by virtue of Lemma \ref{lem:potential-theoretic}, it is enough to show that
\[
I:= \int_{\lambda_0}^{\lambda_*}\int_{\Pi_\gam(\Sigma)} \int_{\R}  \frac{d\nu_\lam^p(x) d\nu_\lam^q(y)}{|x-y|^s} \,d\lam <\infty.
\]
After changing variables, we may rewrite
\[
I = \int_{\lambda_0}^{\lambda_*} \int_\Sigma \int_{\{0,1\}^\N}\frac{d\eta^p(\iii) d\eta^q(\jjj)}{|\Pi_\lam(\iii)-\Pi_\lam(\jjj)|^s} \,d\lam.
\]
Lemma \ref{lem:application-trans} and Fubini yield the estimate
\begin{align*}
I &\le C \int_\Sigma \int_{\{0,1\}^\N} \lambda_0^{-s|\iii\wedge\jjj|} d\eta^p(\iii)d\eta^q(\jjj)\\
&= C \sum_{n=0}^\infty \lam_0^{-s n} (\eta^p\times\eta^q)(\{(\iii,\jjj):\jjj\in \Sigma, |\iii\wedge\jjj|=n\})\\
&= C \sum_{n=0}^\infty \sum_{\iii\in\{0,1\}^n} \lam_0^{-s n} \eta^p([\iii])\eta^q(\Sigma\cap[\iii]) \\
&\le C \sum_{n=0}^\infty \lam_0^{-sn} \max_{ \iii\in\{0,1\}^n: \Sigma\cap [\iii]\neq \varnothing } \eta^p([\iii])\\
&\le C C'  \sum_{n=0}^\infty \lam_0^{-sn}  e^{-n(H_p^q-\e)},
\end{align*}
where we used \eqref{eq:appl-unif-convergence} in the last line. Since, according to the definition of $s$, $\lam_0^s e^{H_p^q-\e} \ge e^{\e}>1$, the last series converges. This completes the proof.
\end{proof}

\section{Proof of Theorem \ref{thm:continuity-mf-spectrum} and further results}

%%%%%%%%%%%%%%%%%%%%%%%%%%%%%%%%%%%%%%%%%%%%%%%%%%%%%%%%%%%%%%%%%%%%%

\subsection{Uniform lower bounds for the local dimension} \label{subsec:uniform-lower-bounds}

For $\lambda=g$, it is known that $d(\nu_{\lambda},x)<1$ for $\nu_{\lambda}$-almost all $x$, and similar results hold for reciprocals of other Pisot numbers \cite{Garsia62}. For the golden ratio, a sharp uniform lower bound for $d(\nu_{\lambda},x)$ was found in \cite{Hu97}. In the general case we can give uniform lower bounds for $\underline{d}(\nu_{\lambda},x)$ which hold for all values of $\lambda$. However in most cases they are substantially less than $1$, and we do not know whether it may happen that $\nu_{\lambda}$ is absolutely continuous but $\underline{d}(\nu_{\lambda},x)<1$ for some $x\in I_\lam$. (Recall that in the biased case, this is possible by Theorem \ref{thm:small-dim-BCs}.)
\begin{theorem} \label{thm:lowerbound}
For any $\lambda\in (\frac{1}{2},1)$ there exist constants $\delta(\lambda),C(\lambda)>0$  such that for all intervals $J\subset I_{\lambda}$,
\[
\nu_{\lambda}(J)\leq C(\lambda)|J|^{\delta(\lambda)}.
\]
In particular, $\min_{x\in I_\lam}\underline{d}(\nu_{\lambda},x)\geq\delta(\lambda)$.

Moreover, for any $k\in\N$, $\delta(\lambda)\rightarrow 1$ as $\lambda\rightarrow 2^{-1/k}$.
\end{theorem}
The first part of the theorem follows from \cite[Proposition 2.2]{FengLau09}. For some specific self-similar measures closely related to the multifractal analysis, it also follows from \cite[Proposition 3.4]{Olsen09}. To prove the second part, we first show that it holds for $k=1$.

\begin{prop} \label{prop:unif-bound-near-ahalf}
We have that $\delta(\lambda)\rightarrow 1$ as $\lambda\rightarrow\frac{1}{2}$.
\end{prop}
\begin{proof}
Given $\iii=i_1\ldots i_n$, we will denote $S_\iii := S_{i_1}\circ \cdots \circ S_{i_n}$. For $k\in\N$, let $$\Lambda_k=\{\lambda:0<2\lambda-1<\lambda^{k-1}(1-\lambda)\}.$$
If $\lambda\in\Lambda_k$, then
$$S_0(I_{\lambda})\cap S_1\circ S_0^{k-2}\circ S_1(I_{\lambda})=\varnothing,$$
and similarly (or by symmetry)
$$S_1(I_{\lambda})\cap S_0\circ S_1^{k-2}\circ S_0(I_{\lambda})=\varnothing.$$
Hence $1 0^{k-1}$ is the only string $\iii$ of length $k$ whose first element is $1$, and such that
$$S_{\iii}(I_{\lambda})\cap S_0(I_\lambda)\neq\varnothing.$$
Similarly, $0 1^{k-1}$ is the only string $\iii$ of length $k$ starting with $0$, and such that
$$S_{\iii}(I_{\lambda})\cap S_1(I_\lambda)\neq\varnothing.$$
We now fix $x\in I_{\lambda}$, and note that there can be at most two distinct $\iii,\jjj\in \{0,1\}^{k}$ such that $x\in S_{\iii}(I_{\lambda})\cap S_{\jjj}(I_{\lambda})$.

Since the intervals are closed, we can choose $\eta>0$ such that for any $x\in I_{\lambda}$ there are still at most two distinct elements $\iii\in \{0,1\}^{k}$ such that $B(x,\eta)\cap S_{\iii}(I_\lambda)\neq\varnothing$. Given $0< r\leq \eta$, we can choose $n$ such that $\lambda^{nk}\eta\leq r\leq \lambda^{(n-1)k}\eta$. If we denote
$$J_n(x)=\{\iii\in \{0,1\}^{kn}:S_{\iii}(I_{\lambda})\cap B(x,\lambda^{(n-1)k}\eta)\neq\varnothing\},$$
then the Bernoulli convolution satisfies
\begin{equation}
\nu_{\lambda}(B(x,r))\leq\frac{\card J_n(x)}{2^{nk}}.  \label{eq:bound-for-BC-using-J}
\end{equation}
We now claim that $\card J_n(x)\leq 2^n$ for all $x$. We have already shown that $J_1(x)\leq 2$ for all $x$, and we now proceed by induction. We assume that for $1\leq l<n$ we have that $J_l(x)\leq 2^l$ for all $x$. Pick $\iii\in J_n(x)$ and write $\iii=\jjj\kkk $, where $|\jjj|=k$ and $|\kkk|=(n-1)k$. Since $S_\jjj(I_\lam)\subset S_\iii(I_\lam)$, we have $S_{\jjj}(I_{\lambda})\cap B(x,\eta)\neq \varnothing$, and therefore there are at most $2$ choices for $\jjj$. Also,
$$S_{\kkk}(I_{\lambda})\cap S_{\jjj}^{-1}B(x,\eta\lambda^{(n-1)k})\neq \varnothing.$$
Note that $|S_{\jjj}^{-1}B(x,\eta\lambda^{(n-1)k})|=\eta\lambda^{(n-2)k}$
and,  by our inductive hypothesis, for $1\leq l<n$ we have that $\#J_l(y)\leq 2^l$ for all $y$. Hence we have that for each $\jjj$ there are at most $2^{n-1}$ choices for $\kkk$. Therefore $\#J_n(x)\leq 2^n$, as claimed.

We now conclude from \eqref{eq:bound-for-BC-using-J} that, for any $x\in I_\lam$,
$$\nu(B(x,r))\leq 2^{-n(k-1)}\leq r^{\frac{n(k-1)\log 2}{nk|\log\lam|+|\log\eta|}} = r^{O(1/n)} r^{\frac{(k-1)\log 2}{k|\log\lam|}}.$$
We have therefore shown that
\[
\text{for all } x\in I_\lam, \quad \underline{d}(\nu,x) \ge \frac{(k-1)\log 2}{k|\log\lam|} \longrightarrow 1 \,\text{ as } \, k\to \infty, \lam\to \frac12,
\]
as claimed.
\end{proof}
This gives the proof of the second part of Theorem \ref{thm:lowerbound} for $k=1$, we now turn to the case $k>1$. We use the following simple lemma.
\begin{lemma}
Let $\nu$ and $\mu$ be measures defined on an interval $I$. Assume there are constants $C_1,C_2>0$ and $\delta_1,\delta_2>\frac{1}{2}$ such that, for any subinterval $J\subset I$, we have $\mu(J)\leq C_1|J|^{\delta_1}$ and $\nu(J)\leq C_2|J|^{\delta_2}$. Then there exists $C_3>0$ such that, for any interval $J_1\subset I+I$, we have $\nu*\mu(J_1)\leq C_3|J_1|^{\delta_1+\delta_2-1}$.
\end{lemma}
\begin{proof}
The product measure $\nu\times\mu$ of any square of side length $r$ is at most $C_1C_2 r^{\delta_1+\delta_2}$. The lemma follows easily since $\nu*\mu$ is the diagonal projection of $\nu\times\mu$.
\end{proof}
We now use the fact that $\nu_{\lambda^{\frac{1}{k}}}$ is the convolution of $k$ scaled copies of $\nu_{\lambda}$ to complete the proof of Theorem \ref{thm:lowerbound}.

\subsection{Proof of Theorem \ref{thm:continuity-mf-spectrum}}

In this section we combine ideas from the proofs of Lemma \ref{lem:freq} and Proposition \ref{prop:unif-bound-near-ahalf} in order to prove Theorem \ref{thm:continuity-mf-spectrum}.
\begin{proof}[Proof of Theorem \ref{thm:continuity-mf-spectrum}]
The proof relies on standard facts on the multifractal spectrum of self-similar measures satisfying the open set condition. See \cite[Chapter 2]{Falconer97} for general background on self-similar measures, and \cite{CawleyMauldin} and \cite{ArbeiterPatzschke96} for their multifractal analysis. In particular, it is well-known that the dimension spectrum $f_{1/2,p_0}(\alpha)$ is the Legendre transform of the function $\tau(q)=\tau_{1/2,p_0}(q)$ given by
\[
\left(p_0^q + (1-p_0)^q\right) 2^{\tau(q)} = 1.
\]
Moreover, $f_{1/2,p_0}$ is also the ``coarse'' multifractal spectrum of $\nu_{1/2}^{p_0}$. See \cite[Chapter 11.1]{Falconer97} for the relevant definitions and proofs.

The proof of the lower bound will involve ideas similar to those in Lemma \ref{lem:freq}. We will fix $k\ge 2$ and $\frac{1}{2}<\lambda<g_k$. As in the proof of Lemma \ref{lem:freq} this means that any sequence in $\{0,1\}^{\N}$ which does not contain $0^k$ or $1^k$ will be in $\mathcal{U}_\lambda$. Let $\Sigma_m$ denote the set of all words in $\{0,1\}^m$ which do not consist only of $0$s or $1$s. If we let $m=\left[\frac{k}{2}\right]$, then the set of functions
\begin{equation} \label{eq:ifs-with-separation}
\{ S_{i_1}\circ\cdots \circ S_{i_m}: (i_1,\ldots,i_m)\in\Sigma_m\}
\end{equation}
 will yield an iterated function system satisfying the strong separation condition. The attractor of the system will be denoted by $\mathcal{A}_{\lambda,k}$, and will be a subset of $\Ak_{g_k}\subset \wtil{\Ak}_\lam$ (recall \eqref{eq:definition-A-lam-tilde} and Lemma \ref{lem:nested-unique-exp}).  Given $\iii=i_1\ldots i_n$, let $p_\iii = p_{i_1}\cdots p_{i_n}$, where $p_0=p$ and $p_1=1-p$. Let $\eta_k$ be the number satisfying
 \[
 \sum_{\iii\in\Sigma_m} p_\iii^{\eta_k} = 1.
 \]
 Clearly, $\eta_k\to 1$ as $k\to \infty$. Let $\wtil{\nu}_{\lam}^p$ be the self-similar measure satisfying
\[
\wtil{\nu}_{\lam}^p = \sum_{\iii\in\Sigma_m} p_\iii^{\eta_k}\,\, (\wtil{\nu}_{\lam}^p\circ S_\iii^{-1}).
\]
It follows from Lemma \ref{lem:meas-when-unique-exp} and the strong separation condition for $\wtil{\nu}_{\lam}^p$ that
 \[
 d(\wtil{\nu}_{\lam}^p,x) = \eta_k\, d(\nula^p,x),
 \]
 for all $x\in \mathcal{A}_{\lambda,k}$ such that the left-hand side is defined. Hence
 \[
 f_{\lam,p}(\alpha) \ge f_{\wtil{\nu}_{\lam}^p}(\eta_k\alpha).
 \]
 However, the multifractal structure of $\wtil{\nu}_{\lam}^p$ is precisely known thanks to the strong separation condition (see e.g. \cite[Theorem 11.5]{Falconer97}). In particular, one can easily check that
\[
f_{\wtil{\nu}_{\lam}^p}(\beta) \to f_{\nu_{1/2}^{p_0}}(\alpha) \quad\text{as}\quad (\lam,p,\beta)\to (\frac12, p_0, \alpha).
\]
Since $\eta_k\to 1$ as $k\to\infty$, we obtain the desired lower bound.

For the upper bound we use the coarse multifractal spectrum: given a finite measure $\mu$ on $\R$, let
\[
\wtil{f}_{\mu}(\alpha) = \lim_{\e\searrow 0}\limsup_{r\searrow 0} \frac{\log(N_\mu(\alpha+\e,\alpha-\e; r)) }{-\log r},
\]
where:
\begin{itemize}
\item $N^+_\mu(\alpha;r)$ is the number of intervals $I_j = B((2j-1)r,r)$ such that $\mu(I_j)\ge r^\alpha$;
\item $N^-_\mu(\alpha;r)$ is the number of intervals $I_j = B((2j-1)r,r)$ such that $\mu(I_j)\le r^\alpha$;
\item $N_\mu(\alpha_1,\alpha_2; r) = \min(N^+_\mu(\alpha_1;r), N^-_\mu(\alpha_2;r))$.
\end{itemize}
It is known that $f_\mu(\alpha)\le \wtil{f}_{\mu}(\alpha)$ for any compactly supported Radon measure, for a proof see \cite[Theorem 3.3.1]{Olsen98}. (Although the definition of coarse multifractal spectrum in \cite{Olsen98} is slightly different---it uses
general packings whereas we use mesh grids---it is a simple exercise to show  that $\widetilde{f}_{\mu}(\alpha)$ will be the same.)
%with equality in the case of self-similar measures with the open set condition. %The reader is again referred to \cite[Chapter 11]{Falconer97} for a proof of %these facts.\footnote{To be precise, in \cite{Falconer97} a slightly different %definition of the coarse spectrum is given and the strong separation condition %is assumed. However, it is not hard to check that the results proven there %remain valid with our definition and the open set condition.}
We will denote $\wtil{f}_{\lam,p} := \wtil{f}_{\nula^p}$.

We will use ideas and notation from Proposition \ref{prop:unif-bound-near-ahalf}. In particular, we let $\Lambda_k$, $\eta$ and $J_n(x)$ be as in that proposition. We fix $k\in\N$ and $\lam\in\Lambda_k$. Without loss of generality, $\lam<2/3$ (the upper bound $2/3$ is arbitrary, any number smaller than $1$ will do).

For $x\in I_{\lambda}$, we will set
$$P_n(x)=\sup\{p_{i_1}\cdots p_{i_{nk}}:(i_1,\ldots,i_{nk})\in J_n(x)\}.$$
Set $r_n=\lambda^{nk}\eta$. Since $\card J_n(x)\le 2^n$,
\[
\text{for any } x\in I_\lam, \quad \nu_{\lambda}^p(B(x,r_n))\leq 2^n P_n(x).
\]

We claim that if $x_j = (2j-1)r_n$ are the centers of the intervals in the $(2r_n)$-mesh of $I_\lam$, then each $\iii\in \{0,1\}^{n k}$ belongs to at most $M=M(k,\eta)$ of the sets $J_n(x_j)$, where $M$ is independent of $n$. Indeed, if $m=\card\{ j: \iii\in J_n(x_j)\}$, then $\iii\in J_n(x')\cap J_n(x'')$ with $|x''-x'|\ge 2(m-1)r_n$, whence, from the definition of $J_n(x)$, and using $\lam<\frac23$,
\[
2\lam^{nk} \ge \frac{\lam}{1-\lam}\lam^{nk} =  \diam(S_\iii(I_\lam)) \ge 2(m-1) \lam^{nk}\eta - 2\lambda^{(n-1)k}\eta,
\]
and $m\le 1 + \lceil 2^k+\eta^{-1}\rceil =: M$, as claimed.

It follows that
\[
\card\{j: \nula^p(B(x_j,r_n)) \ge r_n^\alpha\} \le M \card\{ \iii\in \{0,1\}^{n k}: p_\iii \ge 2^{-n} r_n^\alpha \}.
\]
Since $\{p_\iii\}$ are precisely the $\nu_{1/2}^p$-measures of the $(2^{-nk})$-mesh intervals of $I_{1/2}=[0,1]$, it follows that
\begin{equation} \label{eq:coarse-ineq-1}
N^+_{\nula^p}(\alpha;r_n) \le M\, N^+_{\nu_{1/2}^p}(\beta_{n,k}; 2^{-nk}),
\end{equation}
where, recalling the definition of $r_n$,
\begin{align*}
\beta_{n,k} &= \frac{\log(2^{-n} r_n^\alpha)}{\log(2^{-nk})} \\
&= \frac{|\log_2\lam|\alpha nk+ n+\alpha\log\eta }{nk}\\
&= |\log_2\lam| \alpha + \frac{1}{k} + O(1/n).
\end{align*}

On the other hand, it follows from \eqref{eq:cylinder-inside-ball} that if $\lam^{\ell+1}/(1-\lam) <\eta$, then $B(x,r_n)$ contains the projected cylinder $\Pi_\lam[\iii|nk+\ell]$, where $x=\pi_\lam(\iii)$. Since $\lam\le \frac23$, we can therefore assume that each ball $B(x,r_n)$ contains a cylinder $\Pi_\lam[\iii|nk+L]$, where $L$ depends only on $\eta$ (and hence only on $k$). Letting $p_* = \min(p,1-p)$, and using that $p_{\iii|nk+L} \ge p_{\iii|nk} p_*^L$, we deduce that
\[
\card\{j: \nula^p(B(x_j,r_n)) \le r_n^\alpha\} \le 2^L\card\{ \iii\in\{0,1\}^{nk}: p_\iii \le p_*^{-L} r_n^\alpha \}.
\]
In other words,
\begin{equation}  \label{eq:coarse-ineq-2}
N^{-}_{\nula^p}(\alpha;r_n) \le 2^L\, N^-_{\nu_{1/2}^p}(\gamma_{n,k}; 2^{-nk}),
\end{equation}
where
\begin{align*}
\gamma_{n,k} &= \frac{\log(p_*^{-L} r_n^\alpha)}{\log(2^{-nk})} \\
&= \frac{|\log_2\lam|\alpha nk- L\log_2(p_*)+\alpha\log\eta }{nk}\\
&= |\log_2\lam| \alpha + O(1/n).
\end{align*}

Combining \eqref{eq:coarse-ineq-1} and \eqref{eq:coarse-ineq-2}, we easily obtain that
\[
f_{\lam,p}(\alpha) \le \wtil{f}_{\lam,p}(\alpha) \le \sup_{ \overline{\alpha}\in \left[|\log_2\lam| \alpha,|\log_2\lam| \alpha+\frac{1}{k}\right]} \wtil{f}_{1/2,p}(\overline{\alpha}).
\]
Recalling that  $\wtil{f}_{1/2,p}(\alpha)=f_{1/2,p}(\alpha)$, for which an explicit formula is known (in particular, it is jointly continuous in $p$ and $\alpha$), the desired upper bound is achieved.

\end{proof}

%\begin{remark}
%Using similar ideas, one may prove that the $L^q$ spectrum of $\nula^p$ also converges to the $L^q$ spectrum of $\nu_{1/2}^{p_0}$ as $(\lam,p)\to (\frac{1}{2},p_0)$. It is possible to deduce that, %for $\lam\approx \frac{1}{2}$, the multifractal formalism for $\nula^p$ holds at least approximately.
%\end{remark}

%%%%%%%%%%%%%%%%%%%%%%%%%%%%%%%%%%%%%%%%%%%%%%%%%%%%%%%%%%%%%%%%%%%%%%%

%\bibliographystyle{plain}
%\bibliography{JSS71}

\end{document}